\def\p{{\rm p}}

\def\d{{\partial }}

\def\V{\mathcal{V}}
\def\R{\mathcal{R}}
\def\L{\mathcal{L}}
\def\N{\mathcal{N}}

\def\vstrut{\phantom{\biggm|}}

\def\diag#1{{{\rm diag}(#1)}}
\def\avg#1.{\langle #1 \rangle}

\documentclass[10pt,reqno]{amsart}
     \makeatletter
     \def\section{\@startsection{section}{1}%
     \z@{.7\linespacing\@plus\linespacing}{.5\linespacing}%
     {\bfseries
     \centering
     }}
     \def\@secnumfont{\bfseries}
     \makeatother
\newtheorem{theorem}{Theorem}[section]
\newtheorem{lemma}[theorem]{Lemma}
\newtheorem{proposition}[theorem]{Proposition}
\newtheorem{corollary}[theorem]{Corollary}
\theoremstyle{definition}
\newtheorem{definition}[theorem]{Definition}
\newtheorem{example}[theorem]{Example}
\theoremstyle{remark}
\newtheorem{remark}[theorem]{Remark}
\numberwithin{equation}{section}
\begin{document}

\title{Krawtchouk-Griffiths Systems II: As Bernoulli Systems}
\author{Philip Feinsilver}
\address{Department of Mathematics\\
Southern Illinois University, Carbondale, Illinois 62901, U.S.A.}
\email{phfeins@siu.edu}


 \begin{abstract}
We call Krawtchouk-Griffiths systems, KG-systems, systems of multivariate polynomials orthogonal with respect
to corresponding multinomial distributions. The original Krawtchouk polynomials are orthogonal with respect to a binomial distribution.
Here we present a Fock space construction with raising and lowering operators. The operators of ``multiplication by X" are found in terms
of boson operators and corresponding recurrence relations presented. The Riccati partial differential equations for the differentiation operators,
Berezin transform and associated partial differential equations are found. These features provide the specifications for a Bernoulli system as a
quantization formulation of multivariate Krawtchouk polynomials.
\end{abstract}

\maketitle

\begin{section}{Introduction}
The original paper of Krawtchouk \cite{Kra} presents polynomials orthogonal with respect to a general
binomial distribution and discusses the connection with Hermite polynomials. 
Krawtchouk polynomials are part of the legacy of Mikhail Kravchuk. 
A symposium in honor of his work and memory was held in Kiev and an accompanying volume was produced
that is most highly recommended, Virchenko \cite{V}.\\

Krawtchouk polynomials appear in diverse areas of mathematics and science. 
Applications range from coding theory, \cite{Sl}, to image processing, \cite{Y}.
Multivariable extensions are of interest and the field is very active. We
cite works which have some connection to the approach in this paper.\\

The idea of extending to the multinomial distribution appears in the foundational work of Griffiths \cite{DG,G1,G2}.
Connections with Lie theory have been studied more recently, \cite{GE1,GE2,I,IT,K,Sh} as well as 
from the point of view of harmonic analysis \cite{X1,X2}.\\

Bernoulli systems in one variable are explained in \cite{FS1}, with higher-dimensional Bernoulli systems appearing in \cite{FS2},
where the basic methods of this work appear initially. As a good resource, the Berezin approach was applied to the Schr\"odinger algebra
in \cite{FKS}.\\

An analysis of the connections between orthogonal polynomials and probability distributions via properties of 
their generating function are accomplished in \cite{AKK,KK}.\\

We summarize the contents of this work. Section 2 reviews the binomial case and introduces the matrix approach.
This is followed up with a review of the basics of symmetric representations including the homomorphism and transpose
properties. In \S4, the matrix construction of Krawtchouk polynomial systems is presented. \S\S1-4 are a review
of the basic material in KG-Systems I, \cite{F}.\\

Appell systems and Bernoulli systems are described next. Appell systems essentially turn out to have a generating function in the form
of the exponential of raising operators acting on a vacuum state. Bernoulli systems are Appell systems with orthogonal polynomials
as basis states. The Bernoulli systems provide models for Fock space constructions and for quantization with variables expressed
in operator form. \S6 discusses the form of the observables and lays out the associated constructions of interest, 
such as coherent states and the Leibniz function.
After a review of the multinomial distribution in \S7, in \S8, we identify Krawtchouk polynomials in the context of Bernoulli systems.
Especially, we find the canonical velocity (differentiation) operators and the form of the observables. Using coherent state techniques,
the lowering operators are found via the Leibniz function. This rounds out a description of the Bernoulli system and associated quantities.
To conclude, we find the $X_j$ variables in selfadjoint form and present associated recurrence formulas for the basis Krawtchouk polynomials.

\begin{subsection}{Basic notations and conventions}

In this paper we will be working over $\mathbb{R}$.

\begin{enumerate}

\item We consider polynomials in $d+1$ commuting
variables.\\

\item Multi-index notations for powers. With
$n=(n_0,\ldots,n_d)$, $x=(x_0,\ldots,x_d)$:
$$ x^n = x_0^{n_0}\cdots x_d^{n_d} $$
and the total degree $|n|=n_0+\cdots+n_d$.
Typically $m$ and $n$ will denote multi-indices, with
$i,j,k,\ell$ for single indices. Running indices
may be used as either type, determined from the context.\\

\item We use the following \textsl{summation convention} \\
 repeated Greek indices, e.g., $\lambda$ or $\mu$, are summed from $0$ to $d$. \\

Latin indices $i$, $j$, $k$, run from $1$ to $d$ unless explicitly indicated otherwise and are summed only when explicitly indicated,
preferring $\ell$ for a single index running from 0 to $d$.\\

We will use the notation for standard basis $e_\ell$ as well for shifting multi-indices, e.g. $n\pm e_\ell$ shifts $n_\ell\pm 1$ accordingly.\\

\item For simplicity, we will always denote identity matrices of
the appropriate dimension by $I$. \\

The transpose of a matrix $A$ is denoted $A^\top$. \\

We will use the notation $\mathcal{O}$ to denote a real orthogonal matrix.\\

\item Given $N\ge 0$, $B$ is defined as
the multi-indexed matrix having as its only non-zero entries 
$$B_{mm}= \binom{N}{m}=\frac{N!}{m_0!\ldots m_d!}$$
the multinomial coefficients of order $N$. \\

\item For a tuple of numbers, $\diag{\ldots}$ is the diagonal matrix with the tuple providing the entries forming the main diagonal.\\

\item Expectation with respect to a given underlying distribution is denoted $\avg \cdot .$.
\end{enumerate}
\end{subsection}
\end{section}
\section{Krawtchouk polynomials in one variable and the binomial distribution}
{ Krawtchouk polynomials} may be defined via the generating function
\begin{align*}
(1+pv)^{N-x}(1-qv)^x&=\sum_{0\le k\le N} v^k\,K_k(x,N)\\
\end{align*}
The polynomials $K_k(x,N)$ are orthogonal with respect to the binomial distribution with parameters $N,p$. 
The associated probabilities have the form 
$$\{\, \binom{N}{0}q^Np^0,\ldots,\binom{N}{x}q^{N-x}p^x,\ldots ,\binom{N}{N}q^0p^N\,\} \ .$$
Let's verify this. Setting $G(v)=(1+pv)^{N-x}(1-qv)^x$, we have
\begin{align*}
\avg G(v)G(w).&=\sum_x \binom{N}{x}q^{N-x}p^x(1+pv)^{N-x}(1-qv)^x(1+pw)^{N-x}(1-qw)^x\\
&=(q+qp(v+w)+qp^2vw+p-pq(v+w)+pq^2vw)^N\\
&=(1+pqvw)^N\\
&=\sum_{k=0}^N \binom{N}{k}(pq)^k (vw)^k
\end{align*}
which shows orthogonality and identifies the squared norms
$$\avg K_iK_j.=\delta_{ij}\binom{N}{i}(pq)^i \ .$$
with $0\le i,j\le N$.

\subsection{Matrix formulation}{
Setting $\displaystyle \begin{pmatrix} y_0\\ y_1\end{pmatrix}=\begin{pmatrix} 1&p\\ 1&-q\\ \end{pmatrix}
\begin{pmatrix} v_0\\ v_1\end{pmatrix}\vstrut$ we have 
$$y_0^{N-x} y_1^x=\sum_k v_0^{N-k}v_1^k \Phi_{kx} \ .$$
We call $\Phi$ a (the) \textsl{Kravchuk matrix }.
The rows of the matrix $\Phi$ consist of the values taken on by the corresponding polynomials at the points $x$.
The expression of orthogonality takes the form \\
$$ \Phi B{\rm P} \Phi^\top=BD$$ \\
where $B$ is the diagonal matrix with entries the binomial coefficients $\binom{N}{k}$, the matrix $\rm P$ is diagonal with entries 
$q^{N-k}p^k$ and $D$ is the diagonal matrix with $D_{ii}=(pq)^i$, $0 \le i \le N$. 
}

\section{Symmetric tensor powers}
Given a $(d+1)\times (d+1)$ matrix $A$, the action on the symmetric tensor algebra of the underlying vector space defines 
its ``second quantization" or \textsl{symmetric representation.}  \\

Introduce commuting variables $v_0,\ldots,v_d$. Map
$$y_i=\sum_{j=0}^d A_{ij}v_j$$
The { induced matrix}, $\bar A$, at level (homogeneous degree) $N=n_0+\cdots+n_d$ has entries $\bar A_{mn}$ determined by the expansion
$$ y^m=y_0^{m_0}\cdots y_d^{m_d}=\sum_{n} \bar A_{mn}v^n \ .$$

\begin{remark}
Monomials are ordered according to dictionary ordering with 0 ranking first, followed by $1,2,\ldots,d$. Thus the first column
of $\bar A$ gives the coefficients of $v_0^N$, etc.
\end{remark}

 The map $A\to \bar A$ is at each level a multiplicative homomorphism,
$$\overline{A_1A_2}=\bar A_1\,\bar A_2$$
thus implementing, for each $N\ge 0$,  a representation of the multiplicative semigroup of 
$(1+d)\times(1+d)$ matrices into 
$\vstrut \binom{N+d}{N}\times \binom{N+d}{N}$ matrices
as well as a representation of the group ${\rm GL}(d+1)$ into ${\rm GL}(\binom{N+d}{N})$.

\subsection{Transpose Lemma}
An important lemma is the relation between the induced matrix of $A$ with that of its transpose, $A^\top$. 
\begin{lemma}{\bf  Transpose Lemma.}\\
The induced matrices at each level satisfy
$$ \overline{A^\top}=B^{-1}{\bar A}^\top B \ .$$ 
\end{lemma}
\begin{remark} 
Proofs of the homomorphism property and Transpose Lemma are presented in \cite{F}.
\end{remark}

\begin{remark}{\sl Diagonal matrices. Multinomial distribution.}\\

Note that the $N^{\rm th}$ symmetric power of a diagonal matrix, $D$, is itself diagonal with homogeneous
monomials of the entries of the original matrix along its diagonal. In particular, the trace will be the $N^{\rm th}$ homogeneous	
symmetric function in the diagonal entries of $D$.

\begin{example}
For $V=\left(\begin{matrix}v_{0} & 0 & 0\\0 & v_{1} & 0\\0 & 0 & v_{2}\end{matrix}\right)$ we have in degree 2,
$$\bar V=
\left(\begin{matrix}v_{0}^{2} & 0 & 0 & 0 & 0 & 0\\0 & v_{0} v_{1} & 0 & 0 & 0 & 0\\0 & 0 & v_{0} v_{2} & 0 & 0 & 0\\0 & 0 & 0 & v_{1}^{2} & 0 & 0\\0 & 0 & 0 & 0 & v_{1} v_{2} & 0\\0 & 0 & 0 & 0 & 0 & v_{2}^{2}\end{matrix}\right)$$
and so on.
\end{example}	
Note that the special matrix $B$, diagonal with multinomial coefficients as entries along the diagonal
may be obtained as the diagonal of the induced matrix at level $N$ of the  all 1's matrix.\\

We see that if $\p$ is a diagonal matrix with entries $p_\ell>0$, $0\le \ell \le d$, $\displaystyle \sum_\ell p_\ell=1$, then the diagonal matrix
$$B\overline{\p}$$
yields the probabilities for the corresponding { multinomial distribution}.
\end{remark}

\section{Construction of Krawtchouk polynomial systems}
We start with $\mathcal{O}$, a real orthogonal matrix with the extra condition that
all entries in the first column are positive.	
Form the probability matrix thus
$$\p=\begin{pmatrix} \mathcal{O}_{00}^2&& {}\\&\ddots &\\ && \mathcal{O}_{d0}^2\end{pmatrix}=
\begin{pmatrix}  p_0&& {}\\&\ddots &\\ && p_d\end{pmatrix}$$
row and column indices running from $0$ to $d$. \\

Define
$$A=\frac{1}{\sqrt{\p}}\,\mathcal{O}\,\sqrt{D}$$
where $D$ is diagonal with all positive entries on the diagonal, normalized by requiring $D_{00}=1$.
The essential property satisfied by $A$ is
$$ A^\top \p A=D \ .$$
while observing that the entries of the first column, label 0, are all 1's, i.e. $A_{\ell 0}=1$, $0\le \ell \le d$. \\

\begin{definition} \label{def:kcond}\rm    We say that $A$ satisfies the \textsl{$K$-condition} if
there exists a positive diagonal probability matrix $\p$ and a positive diagonal matrix $D$ such that
$$ A^\top \p A=D \ .$$
with $A_{\ell 0}=1$, $0\le \ell\le d$.
\end{definition}

\begin{paragraph}{\bf Notation}
Throughout the remainder of this work, if $A$ satisfies the $K$-condition, we will denote its inverse by $C$. Thus,
\begin{equation}\label{eq:inv}
C=A^{-1}=D^{-1}A^\top \p \ .
\end{equation}
\end{paragraph}

We note two useful properties
\begin{proposition} \label{prop:x}
For $A$ satisfying the $K$-condition we have\\

1. $(p_0,p_1,\ldots,p_d)A=(1,0,\ldots,0)=e_0$. That is, the vector of probabilities $\{p_\ell\}$ times $A$ yields $e_0$. We express this as
$$p_\mu A_{\mu \ell}=\delta_{0\ell} \ .$$

2. The first row of $C$ is $(p_0,p_1,\ldots,p_d)$, i.e., $C_{0\ell}=p_\ell$.
\end{proposition}
\begin{proof} Start with the observation that 
since the first column of $A$ consists of all 1's, the first row of $A^\top$ is all 1's. So the first row of $A^\top\p$ 
is $(p_0,p_1,\ldots,p_d)$.\\

Now, for \#1, the row of probabilities times $A$ is the first row of $A^\top \p A$, thus, the first row of $D$, which is precisely $e_0$.\\
For \#2, using the form $D^{-1}A^\top \p$ for $C$, as in \#1, the top row of $A^\top \p$ is the row of probabilities, and multiplication by
$D^{-1}$ leaves it unchanged, as $D_{00}=1$.
\end{proof}

\subsection{Krawtchouk systems}{
In any degree $N$, the induced matrix $\bar A$ satisfies
$$ \overline{A^\top}\bar \p\bar A=\bar D \ .$$
Using the { Transpose Lemma}
$$B\overline{A^\top}=\bar A^\top B$$ 
where $B$ is the special multinomial diagonal matrix yields
$$\Phi\, B\bar \p\, \Phi^\top=B\bar D$$
the { Krawtchouk matrix} $\Phi$ being thus defined as $\bar A^\top$. \\

The entries of $\Phi$ are the values of the {\bf  multivariate Krawtchouk polynomials} thus determined. \\

$B\bar D$ is the diagonal matrix of squared norms according to the orthogonality of the Krawtchouk polynomial system with respect
to the corresponding multinomial distribution.
}

\begin{example} 
Start with the orthogonal matrix
$\displaystyle \mathcal{O}=\begin{pmatrix}\sqrt{q}&\sqrt{p}\\ \sqrt{p}&-\sqrt{q}\end{pmatrix}\ .$ \\
Factoring out the squares from the first column yields 
$$\p={\begin{pmatrix}q&0\\ 0&p\end{pmatrix}}$$
and we take
$$A=\begin{pmatrix}1&p\\ 1&-q\end{pmatrix}$$
satisfying
$$A^\top\p A={\begin{pmatrix}1&0\\ 0& pq\end{pmatrix}}=D\ .$$ 
Take $N=4$. 
We have the { Kravchuk matrix} 
$\Phi={\bar A}^\top=$
$$\begin{pmatrix} 1&1&1&1&1\\4p&-q+3p&-2q+2p&-3q+p&-4q\\
6{p}^{2}&-3pq+3{p}^{2}&{q}^{2}-4pq+{p}^{2}&3{q}^{2}-3pq&6{q}^{2}\\
4{p}^{3}&-3{p}^{2}q+{p}^{3}&2p{q}^{2}-2{p}^{2}q&-{q}^{3}+3p{q}^{2}&-4{q}^{3}\\
{p}^{4}&-{p}^{3}q&{p}^{2}{q}^{2}&-p{q}^{3}&{q}^{4}\end{pmatrix} \ .$$
$\p$ is promoted to the induced matrix
$$\bar \p=\begin{pmatrix} q^4&0&0&0&0\\0&q^3p&0&0&0\\0&0&q^2p^2&0&0\\
0&0&0&qp^3&0\\0&0&0&0&p^4\end{pmatrix}\ .$$  
and the binomial coefficient matrix  $B=\diag{1,4,6,4,1}$.
\end{example}

\begin{remark} This approach is presented in detail in \cite{F}.
Here we continue with an analytic approach based on operator calculus techniques.
\end{remark}

\section{Appell and Bernoulli systems}{
An \textsl{Appell system} of polynomials is a sequence $\{\phi_n(x)\}_{n\ge0}$ such that
\begin{enumerate}  
\item $\deg \phi_n=n$ 
\item $\d_x\phi_n=n\phi_{n-1} $  where $\displaystyle \d_x=\frac{d}{dx}$. 
\end{enumerate}\hfill\\
Introduce the \textbf{raising operator}
$$\R\phi_n=\phi_{n+1} \ .$$
The pair $\d_x,\R$ satisfy the commutation relations
$$[\d_x,\R]=I$$
of the Heisenberg-Weyl algebra, i.e., { boson commutation relations}.
}
Consider a convolution family of probability measures $p_t$,
$p_t*p_s=p_{t+s}$, for $s,t\ge0$, $p_0$ a point mass at $0$, with corresponding family of
moment generating functions
$$\int_{\mathbb{R}} e^{zx}\,p_t(dx)=e^{tH(z)}$$
where, extending $z$ to complex values, we assume $H(z)$ to be analytic in a neighborhood
of the origin in $\mathbb{C}$, with $H(0)=0$. 
\begin{remark}
Discrete values of $t$ work in general,
while for continuous $t\ge 0$ we require $p_t$ to be 
infinitely divisible.
\end{remark}
We have as generating function for the sequence $\{\phi_n\}$
$$e^{xz-tH(z)}=\sum_{n=0}^\infty \frac{z^n}{n!}\,\phi_n(x,t)$$
including the additional ``time" variable. Note that $\phi_0(x,t)=1$ with
$$\int_{\mathbb{R}}\phi_n(x,t)\,p_t(dx)=\delta_{0n}$$
for $n\ge0$.

\begin{remark}
In the infinitely divisible case, we have the exponential martingale for the corresponding process with independent increments.
\end{remark}
For the multivariate case,  in the exponent, $xz= \sum x_i z_i$. We have $\d_j=\d/\d x_j$, with $\R_i$ raising the index $n_i$ to $n_i+1$,
satisfying
$$[\d_j,\R_i]=\delta_{ij}1$$
noting that the action of $\d_j$ is the same as multiplication by $z_j$ and that the action of $\R_i$ is the same
as $\d/\d z_i$ .

\subsection{Canonical Appell system}
Now observe that if, in one variable, $V(z)$ is analytic in a neighborhood	of $0\in\mathbb{C}$, we can apply the operator
$V(\d)$ to polynomials in $x$ and we have as well
$$V(\d)\,e^{xz}=V(z)\,e^{xz}$$
acting on exponentials, for $z$ in the domain of $V$. We have further the commutation relation
$$[V(\d),x]=V'(\d)$$
differentiating $V$. Next require that $V(0)=0$, $V'(0)\ne 0$ so that $V$ has a locally analytic inverse in a neighborhood
of the origin as well, denoted by $U(v)$, $U(V(z))=z$. This yields a \textsl{canonical pair}
$$\V=V(\d) \qquad \hbox {and }\qquad \R=xW(\d)$$
where $W(z)=1/V'(z)$.\\

If we have an Appell system in several variables as above, we define
canonical raising and \textsl{velocity} operators defined by
$$\V_j\phi_n=n_j\phi_{n-e_j}\qquad \text{and} \qquad \R_i\phi_n=\phi_{n+e_i}$$
satisfying
$$[\V_j,\R_i]=\delta_{ij}\,1$$
where $\V=(\V_1,\ldots,\V_d)$ is given by a function $V$ of $\d=(\d_1,\ldots,\d_d)$, analytic in a neighborhood of $0$ in $\mathbb{C}^d$, with a locally analytic inverse, $U$. The generating function becomes
$$e^{xz-tH(z)}=\sum_n \frac{V(z)^n}{n!}\,\phi_n(x,t)$$
with multi-index notation for the monomials in $V(z)$, and $n!=n_1!\cdots n_d!$ as usual. The generating function
thus takes the equivalent form
$$e^{xU(v)-tH(U(v))}=\sum_n \frac{v^n}{n!}\,\phi_n(x,t)$$
with multiplication by $v_j$ implemented as the operator $V_j(\d)$ and $\d/\d v_j$ yielding the raising operator $\R_j$
after expressing the action in terms of $\d$.

\begin{example} For an example in one variable, take
$$V(z)=-\log(1-z)\,,\quad U(v)=1-e^{-v}\,,\quad W(z)=1-z$$
with no time variable we have
$$\exp\left(x(1-e^{-v})\right)=\sum_{n\ge 0}\frac{\phi_n(x)}{n!}\,v^n$$
with action of the raising operator
$$\R\phi_n=x(1-\partial)\phi_n=\phi_{n+1}$$
and 
$$\V=-\log(1-\d)=\sum_{n\ge1}\frac{\d^n}{n} \ .$$
The coefficients of the polynomials $\phi_n$ are (up to sign) Stirling numbers of the second kind.\\

With $p(dx)=e^{-x}\,dx$ on $x\ge0$, we have
$$\int_0^\infty e^{zx-x}dx=(1-z)^{-1}=e^{H(z)}$$
so $H(z)=-\log(1-z)$, which happens to equal $V(z)$. We get
$$e^{xz-tH(z)}=\exp\left(x(1-e^{-v})\right)e^{-tv}=\sum_{n\ge 0}\frac{\phi_n(x,t)}{n!}\,v^n$$
where now the raising operator is
$$\R=x(1-\d)-t$$
leaving $\V$ unchanged. And we have for $t>0$,
$$\int_0^\infty e^{-x}x^{t-1}\phi_n(x,t)\,dx/\Gamma(t)=\delta_{0n}$$
for $n\ge 0$, the family of measures $p_t$ given by
$$p_t(dx)=e^{-x}x^{t-1}\,dx/\Gamma(t)$$
on $[0,\infty)$.

\end{example}

\subsection{Bernoulli systems} A Bernoulli system is a canonical Appell system such that, for each $t$, the polynomials $\{\phi_n(x,t)\}$ form an 
orthogonal system with respect to the measure $p_t$.
To indicate this, write $J_n$ generically for the corresponding canonical Appell sequence, thus
$$e^{xU(v)-tH(U(v))}=\sum_{n\ge0}\frac{v^n}{n!}\,J_n(x,t) \ .$$

\begin{example}
Probably the most well-known example are Hermite polynomials, $\{H_n\}$, with generating function
$$e^{xz-z^2t/2}=\sum_{n \ge 0}\frac{z^n}{n!}\,H_n(x,t)$$
orthogonal with respect to the Gaussian distribution with mean zero and variance $t$. Thus $H(z)=z^2/2$,
$$\R=x-t\d\qquad\hbox{and}\qquad \V=\d \ .$$

For an example with nontrivial $V$, consider a family of Poisson-Charlier polynomials with generating function $(1+v)^x\,e^{-tv}$. So
$$U(v)=\log(1+v)\,,\quad V(z)=e^z-1\,,\quad W(z)=e^{-z}$$
with $H(z)=e^z-1$, equal to $V(z)$ in this case. Thus
$$\R=xe^{-\d}-t\qquad\hbox{and}\qquad \V=e^\d-1 \ .$$
The polynomials are orthogonal with respect to the Poisson distribution	on the nonnegative integers with mean $t$.
\end{example}
\subsubsection{ Operator formulation} 
We construct a representation space for the boson commutation relations starting with a \textsl{vacuum state}, $\Omega$,
satisfying $\V_j\Omega=0$, $\forall j$. The basis states are built by acting with the raising operators $\R_i$ on the vacuum state,
thus they are of the form $\R^n\Omega$, for multi-indices $n$, $n_i\ge0$.\\

The  operator form of the generating function is the exponential of the raising operators acting on the vacuum state, using
the abbreviated notation $V(z)\R=\sum\limits_iV_i(z)\R_i$:
$$e^{V(z)\R}\Omega=e^{xz-tH(z)}=\sum_{n\ge0} \frac{V(z)^n}{n!}\,J_n(x,t)$$
where the vacuum state $\Omega$ is here $J_0(x,t)$, the constant function equal to 1. \\

 Introducing the inverse function $ U$, the generating function takes the form
$$e^{v\R}\Omega=e^{xU(v)-tH(U(v))}=\sum_{n\ge0} \frac{v^n}{n!}\,J_n(x,t)$$
with the actions of $\{\R_i\}$ and $\{\V_j\}$ as
$$\R_i J_n=J_{n+e_i}\,,\qquad \V_j J_n=n_j J_{n-e_j} \ .$$
\section{Quantization} We want a commuting family of selfadjoint
operators to serve as quantum observables.
Introduce the operators $X_j$, multiplication by the variables $x_j$.
These will provide the desired operators.\\

Rewrite the generating function in the form
$$e^{zX}\Omega=e^{tH(z)}\,e^{V(z)\R}\,\Omega \ .$$
We start by differentiating with respect to $z_j$ yielding the relation
$$X_j=t\,\frac{\d H}{\d z_j}+\sum_i \R_i\,\frac{\d V_i}{\d z_j} \ .$$
These act as operators by converting the $z_j$ to the partial differentiation operators,
$\d_j$.

\subsection{Specification of the system}
Let's consider the various operators and features involved in specifying the Bernoulli system.\\

First, since we have a Hilbert space (in the present context, over $\mathbb{R}$), we want to find
\textsl{lowering} operators $\{\L_1,\ldots,\L_d\}$, where, for each $i$, $\L_i$ is adjoint to $\R_i$.  
We wish to express all operators in terms of the canonical raising and velocity operators $\R_i$, $\V_j$.  \\

Since we are working with noncommuting operators, it is of interest to study the Lie algebra generated by the raising and lowering operators.\\

With the $\L_j$ in hand, we will express $X_j$ in manifestly self-adjoint form. \\

Some related constructions of interest will be considered as well, notably the Berezin transform, based on the inner product
of coherent states generated by the raising operators. This information is summarized in the \textsl{Leibniz function} to be
explained subsequently. These will provide tools to study the relationships among the lowering operators and the raising
and velocity operators. We will find as well the Riccati partial differential equations satisfied by the velocity operators,
a hallmark feature of Bernoulli systems and equations related to the Leibniz function/Berezin transform.\\

We start in the next two sections reviewing properties of multinomial distributions and the details of the Krawtchouk polynomials
providing the basis states for the Bernoulli system.

\section{Multinomial distribution} \label{sec:multi} First we describe the multinomial process we are interested in.
The process is a counting process keeping track of $d$ possible results, with the possibility that none of them occurs.
 Thus, at each time step the process makes one of $d+1$ choices: \\

1. With probability $p_0$, none of the levels $1$ through $d$ increase. \\

2. With probability $p_i$, $1\le i\le d$, level $i$ increases by 1.  \\

The corresponding { moment generating function} for one time step is
\begin{align*}
p_0+\sum_i p_i e^{z_i}&=1+\sum_i p_i(e^{z_i}-1)\\ 
&=p_\mu e^{z_\mu}
\end{align*}
where we set $z_0=0$. The moment generating function for $N$ steps is thus
$$e^{tH(z)}=(p_\mu e^{z_\mu})^N$$  
where we identify
$$t=N\qquad\hbox{and}\qquad H(z)=\log\left(p_\mu e^{z_\mu}\right) \ .$$

\section{Multivariate Krawtchouk polynomials as Bernoulli systems}
We are given a matrix $A$ satisfying the $K$-condition $A^\top\p A=D$.
The Kravchuk matrix $\Phi$ is the transpose of the symmetric power of $A$.
In degree $N$, we replace the index $m$ by the variables
$\{N-\sum x_i,x_1,\ldots,x_d\}$, where the system has $d$ variables, the variable
$x_0$ being determined by homogeneity, equivalently, in terms of the process, after $N$ steps
if you know $x_1,\ldots,x_d$, then $x_0$ is known. Thus,
$$(Av)^x = \sum_n v^n \Phi_{nx}=\sum_n \frac{v^n}{n!}\, K_n(x,N) \ .$$
More explicitly,
\begin{align*}
(A_{0\mu}v_\mu)^{N-\sum x_i}&(A_{1\mu}v_\mu)^{x_1}\cdots (A_{d\mu}v_\mu)^{x_d}\\ &=\sum \frac{v^n}{n!}\,K_n(x,N) \ .
\end{align*}
Recall that the first column of $A$ consists of all 1's, and set $\alpha_0=A_{00}=1$, 
$\alpha_i=A_{0i}$, { $1\le i\le d$}. \\
We get
$$(\alpha_\lambda v_\lambda)^N\,\prod_i \left(\frac{A_{i\mu}v_\mu}{\alpha_\nu v_\nu}\right)^{x_i}=
e^{xU(v)-NH(U(v))}$$
as the generating function for a Bernoulli system.

\subsection{Identification of Bernoulli constituents}
Let's determine the various Bernoulli parameters.\\

As seen in \S\ref{sec:multi}, we have $t=N$ and $ H(z)=\log p_\mu e^{z_\mu}$ .
From the generating function, the coefficient of $N$ shows that
\begin{equation}\label{eq:H}
H(z)=\log p_\mu e^{z_\mu}=\log(1/\alpha_\mu V_\mu(z))
\end{equation}
or
\begin{equation}\label{eq:y}
p_\mu e^{z_\mu}=\frac{1}{\alpha_\mu V_\mu(z)} 
\end{equation}

Looking at the coefficients of the variables $x_i$ in the exponent, we identify $U$, the inverse to $V$, such that
\begin{equation}\label{eq:x}
U_k(v)=\log \frac{A_{k\mu} v_\mu}{\alpha_\nu v_\nu} 
\end{equation}

\subsection{Canonical velocity operators}Now we can solve for the velocity operators $V_k(z)$.
We have the inverse matrix
$$C=A^{-1}=D^{-1}A^\top \p \ .$$ 
Combine equations \eqref{eq:x} and \eqref{eq:y} to get
$$z_k=U_k(V)=\log\left( {p_\lambda e^{z_\lambda} A_{k\mu} V_\mu}\right)$$
Now exponentiate and move the $p$ factor across 
\begin{equation}\label{eq:z}
\frac{e^{z_k}}{p_\mu e^{z_\mu}}=A_{k\mu} V_\mu
\end{equation}
and applying $C$ to both sides we have
\begin{proposition}\label{prop:vel}
For the Krawtchouk Bernoulli system we have the canonical velocity operators
$$V_k(z)=\frac{1}{p_\mu e^{z_\mu}}\,C_{k\lambda}e^{z_\lambda} \ .$$
satisfying the Riccati partial differential equations
$$\frac{\d V_i}{\d z_j}=\bigl(C_{ij}-p_jV_i\bigr)\,A_{j\mu}V_\mu $$ \\
for {$1\le i,j\le d$}.  
\end{proposition}
\begin{proof}We need only verify the form of the differential equations. We have
\begin{align*}
\frac{\d V_i}{\d z_j}&=-\frac{p_je^{z_j}}{p_\mu e^{z_\mu}}\,V_i+\frac{C_{ij}e^{z_j}}{p_\mu e^{z_\mu}}\\
&=(-p_jV_i+C_{ij})\frac{e^{z_j}}{p_\mu e^{z_\mu}}
\end{align*}
and the result follows upon substituting the relation from equation \eqref{eq:z}.
\end{proof}

 It is convenient to assign/adjoin  \textsl{ projective coordinates}, $v_0=V_0=1$, and we have previously set $z_0=0$. To verify
consistency, substitute $k=0$ in the formula for $V_k$:
$$V_0(z)=\frac{1}{p_\mu e^{z_\mu}}\,C_{0\lambda}e^{z_\lambda} \ . $$
Now invoke Proposition \ref{prop:x}, \#2, to see that $V_0$ is identically equal to one.

\subsection{Observables}
We can now express the observables $X_j$ in terms of the raising and velocity operators.
Recall the relation
$$X_j=t\,\frac{\d H}{\d z_j}+\sum_i \R_i\,\frac{\d V_i}{\d z_j}$$ 
resulting by differentiating the generating function with respect to $z_j$.
\begin{proposition}\label{prop:obs}
The observables $X_j$ have the form
$$X_j=\left(N\,p_j+\sum_i \R_i(C_{ij}-p_j\V_i)\right)\,A_{j\mu} \V_\mu \ .$$ 
\end{proposition}
\begin{proof}Observe that
$$\frac{\d H}{\d z_j}=\frac{p_je^{z_j}}{p_\mu e^{z_\mu}}=p_jA_{j\mu} \V_\mu$$
by equation \eqref{eq:z}. Now apply Proposition \ref{prop:vel} to get the result.
\end{proof}

\section{Coherent states. Leibniz function. Lie algebra}
Now we want to find the lowering operators, the operators adjoint to the raising operators. 
$\L_i$ denotes the adjoint of $\R_i$. We employ techniques involving coherent states.\\

The generating function $e^{V\R}\Omega$ is a type of coherent state. The inner product of coherent states
has the form
$$\Upsilon=\avg e^{B\R}\Omega,e^{V\R}\Omega.=\phi(B_1V_1,\ldots,B_dV_d)$$
by orthogonality.  Working with this we can find the lowering operators. \\

 We have 
$$\Upsilon=\avg \Omega,e^{B\L}e^{V\R}\Omega.$$
equal to the vacuum expectation value of the group element $e^{B\L}e^{V\R}$. 
Comparing with the Heisenberg-Weyl group
$$e^{B\d}e^{VX}=e^{VX}e^{BV}e^{B\d}$$
we call $\Upsilon$ the \textsl{Leibniz function} of the system.

\subsection{Finding the lowering operators}
If we know the Leibniz function, we have the differential relations

$$
\frac{\d \Upsilon}{\d V_i}=\avg e^{B\R}\Omega,\R_i e^{V\R}\Omega. \quad\hbox{and}\quad
\frac{\d \Upsilon}{\d B_i}=\avg e^{B\R}\Omega,\L_i e^{V\R}\Omega. \ .
$$
These are effectively the \textsl{Berezin transforms} of $\R_i$ and $\L_i$ respectively.\\

Thus to find the lowering operators, we wish to express the partial derivatives $\displaystyle \frac{\d \Upsilon}{\d B_i}$ in terms of $V_i$ and 
$\displaystyle \frac{\d \Upsilon}{\d V_i}$. With the correspondence
$$\frac{\d \Upsilon}{\d V_i} \longleftrightarrow \R_i$$
we will have found the lowering operators in terms of the canonical raising and velocity operators.

\subsection{Leibniz function for the Krawtchouk system}
In our case, the generating function for the Krawtchouk polynomials is the coherent state we will use:
$$e^{V\R}\Omega=e^{xU(V)-tH(U(V))}$$
Multiplying by $e^{B\R}\Omega$ and averaging, we have
$$\avg e^{xU(V)+xU(B)}.e^{-t(H(U(B))+H(U(V)))}$$
Recalling the moment generating function, using averaging notation,
$$\avg e^{xz}.=e^{tH(z)}$$
we find in the exponent $t$ times
$$H(U(B)+U(V))-H(U(B))-H(U(V))=\psi(BV)=\psi(B_1V_1,\ldots,B_dV_d)$$
thus defining $\psi$, where we use the fact that we have an orthogonal system.

\begin{proposition}\label{prop:leib}
The Leibniz function for the Krawtchouk system is given by
$$\Upsilon=\avg e^{B\R}\Omega,e^{V\R}\Omega.=(B_\mu D_\mu V_\mu)^N$$
where $B_0=V_0=1$ and $D_i=D_{ii}$ are the diagonal entries of $D$.
\end{proposition}
\begin{proof}We will show that the function $\psi$ above is given by
$$\psi(BV)=\log B_\mu D_\mu V_\mu \ .$$
By equation \eqref{eq:H}, we have
\begin{equation}\label{eq:BV}
H(U(V))=\log(1/\alpha_\mu V_\mu) \qquad\hbox{and}\qquad H(U(B))=\log(1/\alpha_\mu B_\mu) \ .
\end{equation}
Now,  using equations \eqref{eq:H} and \eqref{eq:x}, we have 
\begin{align*}
H(U(V)+U(B))&=\log(p_\mu e^{U_\mu(V)}e^{U_\mu(B)})\\
&=\log\left(p_\mu  \frac{A_{\mu\lambda} V_\lambda}{\alpha_\sigma V_\sigma} \frac{A_{\mu\nu} B_\nu}{\alpha_\epsilon B_\epsilon}\right)
\end{align*}
Rewriting this last in the form
$$\log\left( \frac{ V_\lambda(A^\top \p A)_{\lambda \nu}B_\nu}{\alpha_\sigma V_\sigma \,\alpha_\epsilon B_\epsilon}\right)$$
invoke the $K$-condition, $A^\top\p A=D$ and bring in equation \eqref{eq:BV} yielding
$$\log(V_\lambda D_\lambda B_\lambda)+H(U(V))+H(U(B))$$
from which the form of $\psi$ follows. Exponentiating and raising to the power $N$ then gives the result.
\end{proof}

\subsection{Lowering operators for the Krawtchouk system. Lie algebra.}
We are now in a position to determine the lowering operators $\L_i$.
\begin{proposition} The Leibniz function $\Upsilon$ satisfies the partial differential equations
$$\frac{1}{D_i}\,\frac{\d \Upsilon}{\d B_i}=NV_i\Upsilon -V_i\sum V_j\frac{\d \Upsilon}{\d V_j} \ .$$ 
\end{proposition}
\begin{proof}
First, for the left hand side
$$\frac{1}{D_i}\,\frac{\d \Upsilon}{\d B_i}=\frac{N V_i}{B_\mu V_\mu D_\mu }\Upsilon \ .$$
Now calculate
$$\sum_j V_j\frac{\d \Upsilon}{\d V_j}=N\sum_j \frac{B_jD_jV_j}{B_\mu V_\mu D_\mu}\Upsilon
=N\,\frac{B_\nu V_\nu D_\nu-1}{B_\mu V_\mu D_\mu}\Upsilon=N(1-\frac{1}{B_\mu V_\mu D_\mu})\Upsilon$$
taking out the term $B_0V_0D_0=1$. Hence
$$N\Upsilon-\sum_j V_j\frac{\d \Upsilon}{\d V_j}=\frac{N}{B_\mu V_\mu D_\mu}\Upsilon$$
and multiplying through by $V_i$ yields the result.
\end{proof}
Re-interpreting the derivatives $\d\Upsilon/\d V_i$ as raising operators $\R_i$ yields
\begin{corollary}\label{cor:L}
The lowering operators have the form
$$\L_i=D_i\bigl(N-\sum_{j=1}^d \R_j\V_j\bigr)\,\V_i \ .$$
\end{corollary}
\begin{subsubsection}{Lie algebra}
Introduce the \textsl{number operator} $\N=\sum_k\R_k\V_k$, 
satisfying
$$\N \R^n\Omega=|n| \R^n\Omega \ .$$
We can write the above result in a convenient form.
\begin{proposition}\label{prop:num}
 In terms of the number operator $\N$, we have
$$\L_i=D_i(N-\N)\V_i \ .$$
\end{proposition}
Note the commutation relations
$$[\N,\R_i]=\R_i \qquad\hbox{and}\qquad [\V_j,\N]=\V_j \ .$$
Next, form the $d^2$ operators 
$$ \rho_{ij}=[\L_i,\R_j] \ . $$
\begin{proposition} \label{prop:rho}
We have\\

1. $\rho_{ii}=D_i(N-\R_i\V_i-\N)$.\\

2. For $i\ne j$, $\rho_{ij}=-D_i\R_j\V_i$.
\end{proposition}
\begin{proof}If $i\ne j$, then $\R_j$ and $\V_i$ commute so that
$$[\L_i,\R_j]=-D_i[\N,\R_j]\V_i=-D_i\R_j\V_i$$
as stated. For $i=j$, we get
$$[\L_i,\R_i]=D_i(N-[\N,\R_i]\V_i-\N)=D_i(N-\R_i\V_i-\N)$$
as required.
\end{proof}
 Denoting adjoint by ${}^*$ we note that
$$\rho_{ij}*=\rho_{ji}$$
and that $\rho_{ii}$, $\N$, and $\R_i\V_i$ are all selfadjoint.\\

For a dimension count, we have $d^2$ operators $\rho_{ij}$
plus the $2d$ raising and lowering operators which yields a Lie algebra of dimension $d^2+2d=(d+1)^2-1$. Thus, we have
a copy of $\mathfrak{sl}(d+1)$.
\end{subsubsection}

\section{Observables}
Going back to the observables, we can express the operators $X_j$ in manifestly selfadjoint form.
\begin{proposition} We have $X_j=$
$$ \sum_{1\le i\le d} (\R_i+\L_i)\,C_{ij}+(N-\N)
-\frac{1}{p_j}\,\sum_{\substack{1\le i\le d\\ 1\le k\le d}}C_{ij}C_{kj}\rho_{ik} $$
for $1\le j\le d$.
\end{proposition}
First we need some basic identities
\begin{lemma}\label{lem:id}
With $C=A^{-1}$, we have
$$p_jA_{ji}=D_iC_{ij}$$
\end{lemma}
which is an explicit form of the matrix relation $DC=A^\top\p$, cf. equation \eqref{eq:inv}.

\begin{proof}
Recall, Proposition \ref{prop:obs},
$$X_j=\left(N\,p_j+\sum_i \R_i(C_{ij}-p_j\V_i)\right)\,A_{j\mu} \V_\mu$$ 
and note that
$$A_{j\mu} \V_\mu=A_{j0}\V_0+\sum_{k=1}^d A_{jk}\V_k=1+\sum_k A_{jk}\V_k \ .$$
We get, using the above Lemma, and Proposition \ref{prop:num},
\begin{align}\label{eq:xx}
(p_j(N-\N)&+\sum_i \R_iC_{ij})(\sum_k A_{jk}\V_k+1) \nonumber \\
&=p_j(N-\N)+(N-\N)\sum_iD_iC_{ij}\V_i+(\sum_i \R_iC_{ij})(\sum_k A_{jk}\V_k+1)\nonumber \\
&=\sum_i(\R_i+\L_i)C_{ij}+p_j(t-\N)+\sum_{i,k}C_{ij}A_{jk}\R_i\V_k \ .
\end{align}
In the last sum, for $i\ne k$, we have $\R_i\V_k=-(1/D_k)\rho_{ki}$. We get
\begin{equation}\label{eq:ik}
-\sum_{i\ne k}\frac{1}{D_k}\,C_{ij}A_{jk}\rho_{ki}=-\frac{1}{p_j}\sum_{i\ne k}\,C_{ij}C_{kj}\rho_{ki}
\end{equation}
as in Lemma \ref{lem:id}. For $i=k$, we have from Proposition \ref{prop:rho},
$$\R_i\V_i=N-\N-\frac{1}{D_i}\,\rho_{ii}$$
and
\begin{align}\label{eq:ii}
\sum_i C_{ij}A_{ji}\R_i\V_i&=\sum_i C_{ij}A_{ji}(N-\N-\frac{1}{D_i}\,\rho_{ii}) \nonumber\\
&=\sum_i C_{ij}A_{ji}(N-\N)-\frac{1}{p_j}\,\sum_i C_{ij}C_{ij}\rho_{ii} \ .
\end{align}
Finally, observe that
$$\sum_i C_{ij}A_{ji}=C_{\mu j}A_{j\mu}-C_{0j}A_{j0}=1-p_j$$
recalling Proposition \ref{prop:x}. Combining equations \eqref{eq:xx}, \eqref{eq:ik} and \eqref{eq:ii} we arrive at the
desired form.
\end{proof}
\subsection{Recurrence formulas}
Now returning to the form of the $X_j$ in terms of the canonical raising and velocity operators
$$X_j=\bigl(N\,p_j+\sum_i \R_i(C_{ij}-p_j\V_i)\bigr)\,A_{j\mu} \V_\mu$$ 
we see that these yield recurrence formulas for the multivariate Krawtchouk polynomial system.
\begin{proposition}The Krawtchouk polynomials satisfy the following recurrence relations\\
\begin{align*}
x_j\,&K_n(x,N)=\\
&p_j(N-|n|)\,K_n+\sum_i C_{ij}K_{n+e_i}+p_jN\sum_k A_{jk}n_k\,K_{n-e_k}\\
&\qquad -p_j\sum_kA_{jk}(|n|-1)n_kK_{n-e_k}+\sum_{i,k}C_{ij}A_{jk}n_k\,K_{n-e_k+e_i} \ .
\end{align*}
\begin{proof} Starting with $A_{j\mu} \V_\mu=1+\sum_k A_{jk}\V_k\vstrut$ as in the above proof,
expand out the formula for $X_j$. Applying to $K_n$ yields the result.
\end{proof}
\end{proposition}

\begin{section}{Conclusion}
The Krawtchouk polynomials and their multivariable generalizations provide models involving
a wide spectrum of mathematical objects. Among the most interesting aspects from the present point of view
are the representations of Lie algebras on spaces of polynomials and the quantization
aspects, especially the connections with quantum probability.\\

The first part of this work, KG-Systems I, emphasizes the linear algebra and numerical aspects of these systems,
while KG-Systems II, shows these systems in the analytic setting of a discrete quantum system, albeit over the reals.\\

One expects these classes of polynomials to be useful in coding theory while applications
in image compression as well as quantum computation would certainly not be wholly unexpected. And from the
theoretical point of view how these systems behave under various limit theorems analogous to the classical Poisson and Central
limit theorems provides an area for interesting further study.\\
\end{section}

\paragraph{\bf Acknowledgment. }
The author is grateful to Prof. Obata and the Center at Tohoku University 
for the opportunity to have participated  in the 
first GSIS-RCPAM International Symposium, Sendai, 2013, where this material
was presented. Fruitful discussions with Prof. Tanaka are appreciatively acknowledged.


\begin{thebibliography}{99}

\bibitem{AKK}
Nobuhiro Asai, Izumi Kubo, and Hui-Hsiung Kuo.
\newblock Generating function method for orthogonal polynomials and
  {J}acobi-{S}zeg{\H o}\ parameters.
\newblock In {\em Quantum probability and infinite dimensional analysis},
  volume~18 of {\em QP--PQ: Quantum Probab. White Noise Anal.}, pages 42--55.
  World Sci. Publ., Hackensack, NJ, 2005.

\bibitem{DG}
Persi Diaconis and Robert Griffiths.
\newblock An introduction to multivariate {K}rawtchouk polynomials and their
  applications.
\newblock {\em J. Statist. Plann. Inference}, 154:39--53, 2014.

\bibitem{F}
Philip Feinsilver.
\newblock Krawtchouk-Griffiths	Systems I: Matrix Approach.
\newblock  Preprint {\tt http://arxiv.org/abs/1611.06991}, 2016.

\bibitem{FKS}
Philip Feinsilver, Jerzy Kocik, and Ren{\'e} Schott.
\newblock Berezin quantization of the {S}chr\"odinger algebra.
\newblock {\em Infin. Dimens. Anal. Quantum Probab. Relat. Top.}, 6(1):57--71,
  2003.

\bibitem{FS1}
Philip Feinsilver and Ren{\'e} Schott.
\newblock {\em Algebraic structures and operator calculus. {V}ol. {I}}, volume
  241 of {\em Mathematics and its Applications}.
\newblock Kluwer Academic Publishers Group, Dordrecht, 1993.
\newblock Representations and probability theory.

\bibitem{FS2}
Philip Feinsilver and Ren{\'e} Schott.
\newblock {\em Algebraic structures and operator calculus. {V}ol. {III}},
  volume 347 of {\em Mathematics and its Applications}.
\newblock Kluwer Academic Publishers Group, Dordrecht, 1996.
\newblock Representations of Lie groups.

\bibitem{GE1}
Vincent~X. Genest, Sarah Post, Luc Vinet, Guo-Fu Yu, and Alexei Zhedanov.
\newblock {$q$}-rotations and {K}rawtchouk polynomials.
\newblock {\em Ramanujan J.}, 40(2):335--357, 2016.

\bibitem{GE2}
Vincent~X. Genest, Luc Vinet, and Alexei Zhedanov.
\newblock The multivariate {K}rawtchouk polynomials as matrix elements of the
  rotation group representations on oscillator states.
\newblock {\em J. Phys. A}, 46(50):505203, 24, 2013.

\bibitem{G1}
R.~C. Griffiths.
\newblock Orthogonal polynomials on the multinomial distribution.
\newblock {\em Austral. J. Statist.}, 13:27--35, 1971.

\bibitem{G2}
R.~C. Griffiths.
\newblock A characterization of the multinomial distribution.
\newblock {\em Austral. J. Statist.}, 16:53--56, 1974.

\bibitem{I}
Plamen Iliev.
\newblock A {L}ie-theoretic interpretation of multivariate hypergeometric
  polynomials.
\newblock {\em Compos. Math.}, 148(3):991--1002, 2012.

\bibitem{IT}
Plamen Iliev and Paul Terwilliger.
\newblock The {R}ahman polynomials and the {L}ie algebra {$\mathfrak{sl}_3({\mathbb{C}})$}.
\newblock {\em Trans. Amer. Math. Soc.}, 364(8):4225--4238, 2012.

\bibitem{K}
Tom~H. Koornwinder.
\newblock Krawtchouk polynomials, a unification of two different group
  theoretic interpretations.
\newblock {\em SIAM J. Math. Anal.}, 13(6):1011--1023, 1982.

\bibitem{KK}
Izumi Kubo and Hui-Hsiung Kuo.
\newblock M{RM}-applicable orthogonal polynomials for certain hypergeometric
  functions.
\newblock {\em Commun. Stoch. Anal.}, 3(3):383--406, 2009.

\bibitem{Kra}
M. Krawtchouk,
\newblock Sur une g\'en\'eralisation des polyn\^omes d'Hermite.
\newblock {\em Comptes Rendus}, 189:620--622, 1929.

\bibitem{Sh}
Genki Shibukawa.
\newblock Multivariate {M}eixner, {C}harlier and {K}rawtchouk polynomials
  according to analysis on symmetric cones.
\newblock {\em J. Lie Theory}, 26(2):439--477, 2016.

\bibitem{Sl}
N.~J.~A. Sloane.
\newblock Error-correcting codes and invariant theory: new applications of a
  nineteenth-century technique.
\newblock {\em Amer. Math. Monthly}, 84(2):82--107, 1977.

\bibitem{V}
N.\,Virchenko, et al., eds.
\newblock {\em Development of the Mathematical Ideas of Mykhailo Kravchuk (Krawtchouk)}.
\newblock Shevchenko Scientific Society, Kyiv-New York, 2004.

\bibitem{X1}
Yuan Xu.
\newblock Tight frame with {H}ahn and {K}rawtchouk polynomials of several
  variables.
\newblock {\em SIGMA Symmetry Integrability Geom. Methods Appl.}, 10:Paper 019,
  19, 2014.

\bibitem{X2}
Yuan Xu.
\newblock Hahn, {J}acobi, and {K}rawtchouk polynomials of several variables.
\newblock {\em J. Approx. Theory}, 195:19--42, 2015.

\bibitem{Y}
Pew-Thian Yap, Raveendran Paramesran, and Seng-Huat Ong.
\newblock Image analysis by {K}rawtchouk moments.
\newblock {\em IEEE Trans. Image Process.}, 12(11):1367--1377, 2003.

\end{thebibliography}
\end{document}